\documentclass[11pt,dvipsnames,a4paper]{article}
\usepackage[affil-it]{authblk}
\usepackage[utf8]{inputenc}
\usepackage{amsmath, amssymb, mathrsfs, amsthm, amsfonts}
\usepackage{latexsym}
\usepackage{hyperref}
\usepackage{a4wide}
\usepackage{color,xcolor}
\usepackage{graphicx}
\usepackage{fancyhdr}

\newtheorem{theorem}{Theorem}[section]
\newtheorem{corollary}[theorem]{Corollary}
\newtheorem{lemma}[theorem]{Lemma}
\newtheorem{proposition}[theorem]{Proposition}

\theoremstyle{definition}

\theoremstyle{remark}
\newtheorem*{example}{Example}

\title{The Varchenko Determinant of an Oriented Matroid}

\author{Hery Randriamaro
\thanks{This research was funded by my mother \\
Lot II B 32 bis Faravohitra, 101 Antananarivo, Madagascar \\
e-mail: \texttt{hery.randriamaro@outlook.com}}}

\begin{document}

\thispagestyle{empty}

\noindent {\color{MidnightBlue} \rule{\linewidth}{2pt}}

\vspace*{15pt}

\noindent \textsf{\huge The Varchenko Determinant of an Oriented Matroid}

\vspace*{25pt}

\noindent \textbf{\Large Hery Randriamaro${}^1$} 

\vspace*{10pt}

\noindent ${}^1$\: This research was funded by the mother of the author \\ \textsc{Lot II B 32 bis Faravohitra, 101 Antananarivo, Madagascar} \\ \texttt{hery.randriamaro@outlook.com} 

\vspace*{25pt}

\noindent \textsc{\large Abstract} \\ 
Varchenko introduced in 1993 a distance function on the chambers of a hyperplane arrangement that gave rise
to a determinant whose entry in position $(C, D)$
is the distance between the chambers $C$ and $D$, and computed that determinant. In 2017, Aguiar and Mahajan provided a generalization of that distance function, and computed the corresponding determinant. This article extends their distance function to the topes of an oriented matroid, and computes the determinant thus defined. Oriented matroids have the nice property to be abstractions of some mathematical structures including hyperplane and sphere arrangements, polytopes, directed graphs, and even chirality in molecular chemistry. Independently and with another method, Hochstättler and Welker also computed in 2019 the same determinant. 

\vspace*{10pt}


\vspace*{10pt}


\vspace*{10pt}

\noindent {\color{MidnightBlue} \rule{\linewidth}{2pt}}

\vspace*{10pt}

\section{Introduction}

\noindent Varchenko introduced and computed a determinant defined from hyperplane arrangements \cite[§~1]{Va}. That determinant already appeared earlier in the implicit form of that of a symmetric bilinear form on a Verma module over a $\mathbb{C}$-algebra \cite[§~1]{ScVa}. It also plays a key role to prove the realizability of variant models of quon algebras in quantum mechanics \cite[Th.~4.2]{Ra}, \cite[Prop.~2.1]{Ra2}. Since then, there have been some attempts to provide cleaner proofs of the result of Varchenko, and even more refined determinants but still defined from hyperplane arrangements \cite[§~8]{AgMa}, \cite[Th.~4.5]{Ge}, \cite[Th.~1.3]{Ra3}. Besides, the associated matrix has been investigated from several angles. For hyperplane arrangements in semigeneral position, Gao and Zhang computed its diagonal form \cite[Th.~2]{GaZh}, and for braid arrangements, Denham and Hanlon studied its Smith normal form \cite[Th.~3.3]{DeHa} while Hanlon and Stanley computed its nullspace \cite[Th.~3.3]{HaSt}. \emph{This article studies the extension of the Varchenko determinant to oriented matroids}. Folkman began to work on oriented matroids by 1967, but unfortunately died before publishing his theory. Lawrence completed that theory, and published the results in a joint paper with Folkman in 1978 \cite{FoLa}. Independently, Bland and Las Vergnas developed the same notion of oriented matroids \cite{BlLas}. Those latter are abstractions for different mathematical objects such as directed graphs \cite[§~1.1]{BjLaStWhZi} and convex polytopes \cite[§~9]{BjLaStWhZi}.

\smallskip

\noindent Let $\mathbb{E} := \{+,-,0\}$, and for $n \in \mathbb{N}^*$ denote by $u_i$ the $i^{\text{th}}$ component of a covector $u$ in $\mathbb{E}^n$. Equip $\mathbb{E}^n$ with the binary operation $\circ$ as follows: if $u, v \in \mathbb{E}^n$, $u \circ v$ is the vector $w$ of $\mathbb{E}^n$ such that $\displaystyle w_i := \begin{cases} u_i & \text{if}\ u_i \neq 0, \\ v_i & \text{otherwise} \end{cases}$. Moreover, define the function $\mathrm{S}: \mathbb{E}^n \times \mathbb{E}^n \rightarrow 2^{[n]}$ by $\mathrm{S}(u,v) := \big\{i \in [n]\ \big|\ u_i = -v_i \neq 0\big\}$. We use the covector axioms for the definition of an oriented matroid, that is, an \textbf{oriented matroid} is a subset $\mathsf{M}$ of $\mathbb{E}^n$ satisfying the conditions
\begin{itemize}
	\item $(0, \dots, 0) \in \mathsf{M}$,
	\item if $u \in \mathsf{M}$, then $-u \in \mathsf{M}$,
	\item if $u,v \in \mathsf{M}$, then $u \circ v \in \mathsf{M}$,
	\item for each pair $u,v \in \mathsf{M}$ and every $j \in \mathrm{S}(u,v)$, there exists $w \in \mathsf{M}$ such that
	$$w_j = 0 \quad \text{and} \quad \forall i \in [n] \setminus \mathrm{S}(u,v):\ w_i = (u \circ v)_i.$$
\end{itemize}

\noindent An integer $i \in [n]$ is a \textbf{loop} for an oriented matroid $\mathsf{M}$ if $u_i = 0$ for every $u \in \mathsf{M}$. \emph{For simplicity, we assume that all oriented matroids considered in this article are loop-free}. The covector set $\mathsf{M}$ forms a poset with the partial order $\leq$ defined by $u \leq v \, \Longleftrightarrow \, \forall i \in [n]:\, u_i \in \{0, v_i\}$. The \textbf{rank} of an oriented matroid $\mathsf{M}$ is the length of a maximal chain in $(\mathsf{M}, \leq)$. A covector $u$ of $\mathsf{M}$ is called a \textbf{tope} if $u_i \neq 0$ for every $i \in [n]$. Denote the set formed by the topes of $\mathsf{M}$ by $T_{\mathsf{M}}$. The \textbf{topal fiber} of $\mathsf{M}$ relative to a set $I \subseteq [n]$ and a covector $u \in \mathsf{M}$ such that $u_i \neq 0$ for every $i \in [n] \setminus I$ is the set of covectors $\mathsf{M}_{I,u} := \{v \in \mathsf{M}\ |\ \forall i \in [n] \setminus I:\, v_i = u_i\}$. Let $R_n$ be the polynomial ring $\mathbb{Z}\big[a_i^+, a_i^-\ \big|\ i \in [n]\big]$ with variables $a_i^+, a_i^-$. Define the \textbf{Aguiar-Mahajan distance} $\mathrm{v}: T_{\mathsf{M}} \cap \mathsf{M}_{I,u} \times T_{\mathsf{M}} \cap \mathsf{M}_{I,u} \rightarrow R_n$ on the topes in $\mathsf{M}_{I,u}$ by
$$\mathrm{v}(v,v) = 1 \quad \text{and} \quad \mathrm{v}(v,w) = \prod_{i \in \mathrm{S}(v,w) \cap I} a_i^{v_i}\ \text{if}\ v \neq w.$$
Remark that the name distance is kept as the authors originally called it so for central hyperplane arrangements $\cite[§~8.1]{AgMa}$. The \textbf{Varchenko determinant} of a topal fiber $\mathsf{M}_{I,u}$ is $$V_{\mathsf{M}_{I,u}} := \big|\mathrm{v}(v,w)\big|_{v,w \in T_{\mathsf{M}} \cap \mathsf{M}_{I,u}}.$$ It naturally becomes the Varchenko determinant of the oriented matroid $\mathsf{M}$ if $I = [n]$. The \textbf{weight} of a covector $u \in \mathsf{M} \setminus T_{\mathsf{M}}$ is the monomial $\displaystyle \mathrm{b}_u := \prod_{\substack{i \in [n] \\ u_i = 0}} a_i^+ a_i^-$. For $i \in [n]$, define the $i^{\text{th}}$ \textbf{boundary} $\partial_i v$ of a tope $v$ of $\mathsf{M}$ by the set $\{u \in \mathsf{M}\ |\ u \leq v,\, u_i = 0\}$. The \textbf{multiplicity} of a covector $u \in \mathsf{M} \setminus T_{\mathsf{M}}$ such that $u_i = 0$ is the integer $\displaystyle \beta_u := \frac{\#\{t \in T_{\mathsf{M}}\ |\ \max \partial_i t = u\}}{2}$. We will see in Theorem~\ref{V} that $\beta_u$ is independent of the chosen $i$. \emph{We can now state the main result of this article}.

\begin{theorem}  \label{ThMain}
	Let $\mathsf{M}$ be an oriented matroid in $\mathbb{E}^n$, $I$ a subset of	$[n]$, and $u$ a covector of $\mathsf{M}$ such that $u_i \neq 0$ for every $i \in [n] \setminus I$. Then, the Varchenko determinant of the topal fiber $\mathsf{M}_{I,u}$ is $$V_{\mathsf{M}_{I,u}} = \prod_{v \in \mathsf{\mathsf{M}_{I,u}} \setminus T_{\mathsf{M}}} (1 - \mathrm{b}_v)^{\beta_v}.$$
\end{theorem}

\noindent For an oriented matroid and under the condition $a_i^+ = a_i^-$, we recover the determinant computed by Hochstättler and Welker with tools from oriented matroid topology \cite[Th.~1]{HoWe}. Besides, Olzhabayev and Zhang obtained $V_{\mathsf{M}_{I,u}}$, for $(\mathsf{M}_{I,u}, \leq)$ being poset isomorphic to a configuration of pseudolines, from the diagonal form of the associated matrix \cite[Th.~1.2]{OlZh}.

\smallskip

\noindent This article is organized as follows: We begin by recalling the topological representation theorem in Section~\ref{SeTop}. That theorem allows to obtain the Varchenko determinant of an oriented matroid from that of a pseudohyperplane arrangement. To compute that latter, we are inspired by the proof strategy of Aguiar and Mahajan for central hyperplane arrangements \cite[Th.~8.11]{AgMa} by extending a Witt identity to pseudohyperplane arrangements in Section~\ref{SeWi}, and by using that extension to determine the Varchenko determinant of a pseudohyperplane arrangement in Section~\ref{SeVa}. We finally prove Theorem~\ref{ThMain} in Section~\ref{SeMain}.

\section{Topological Representation Theorem} \label{SeTop}

\noindent We describe the topological representation theorem, and the Varchenko determinant of a pseudohyperplane arrangement. For $d \in \mathbb{N}^*$, a \textbf{subsphere} of a sphere $\mathbb{S}^d$ is a subset of $\mathbb{S}^d$ which is homeomorphic to $\mathbb{S}^{d-1}$.

\begin{lemma}\cite[Lem.~5.1.1]{BjLaStWhZi} \label{LeTa}
Letting $d \in \mathbb{N}^*$, for a subsphere $S$ of $\mathbb{S}^d$ the following conditions are equivalent:
\begin{enumerate}
\item $S$ is homeomorphic to the equator $\big\{(x_1, \dots, x_{d+1}) \in \mathbb{S}^d\ |\ x_{d+1}=0\big\}$,
\item $S$ is homeomorphic to some piecewise-linearly embedded subsphere,
\item the closure of each connected component of $\mathbb{S}^d \setminus S$ is homeomorphic to the $d$-ball.
\end{enumerate}
\end{lemma}

\noindent A subsphere of $\mathbb{S}^d$ satisfying the conditions in Lemma~\ref{LeTa} is called a \textbf{pseudosphere} of $\mathbb{S}^d$. For a pseudosphere $S$ of $\mathbb{S}^d$, denote by $S^+$ and $S^-$ both connected components of $\mathbb{S}^d \setminus S$ called \textbf{sides} of $S$. A finite set $\mathcal{S} = \{S_i\}_{i \in [n]}$ of pseudospheres in $\mathbb{S}^d$ is called a \textbf{pseudosphere arrangement} if the following conditions hold:
\begin{enumerate}
\item $\displaystyle S_I := \bigcap_{i \in I}S_i$ is a sphere for all $I \subseteq [n]$.
\item For $I \subseteq [n]$ and $i \in [n]$, if $S_I \nsubseteq S_i$, then $S_I \cap S_i$ is a pseudosphere in $S_I$ with sides $S_I \cap S_i^+$ and $S_I \cap S_i^-$.
\item The intersection of an arbitrary set of closed sides in $\mathbb{S}^d$ is either a sphere or a ball.
\end{enumerate}

\noindent For a pseudosphere arrangement $\mathcal{S} = \{S_i\}_{i \in [n]}$, define the sign functions $\sigma_i: \mathbb{S}^d \rightarrow \mathbb{E}$ and $\sigma_{\mathcal{S}}: \mathbb{S}^d \rightarrow \mathbb{E}^n$ respectively, for $x \in \mathbb{S}^d$, by $$\sigma_i(x) := \begin{cases}
+ & \text{if}\ x \in S_i^+ \\
- & \text{if}\ x \in S_i^- \\
0 & \text{if}\ x \in S_i
\end{cases} \quad \text{and} \quad \sigma_{\mathcal{S}}(x) := \big(\sigma_1(x), \dots, \sigma_n(x)\big).$$

\noindent Consider a pseudosphere arrangement $\mathcal{S} = \{S_i\}_{i \in [n]}$ in $\mathbb{S}^d$. Its \textbf{rank} is $\displaystyle d - \dim \bigcap_{i \in [n]}S_i$ assuming that $\dim \emptyset = -1$. It is \textbf{essential} if $\displaystyle \bigcap_{i \in [n]}S_i = \emptyset$, and \textbf{centrally symmetric} if $-S_i = S_i$ for every $i \in [n]$. Besides, denote by $\mathsf{M}_{\mathcal{S}}$ the subset $\big\{\sigma_{\mathcal{S}}(x)\ |\ x \in \mathbb{S}^d\big\} \cup \big\{(0, \dots, 0)\big\}$ of $\mathbb{E}^n$.

\begin{theorem}\cite[Topological Representation Theorem~5.2.1]{BjLaStWhZi}
Let $d$ be a nonnegative integer and $\mathsf{M}$ a subset of $\mathbb{E}^n$. Then the following conditions are equivalent:
\begin{enumerate}
\item $\mathsf{M}$ is an oriented matroid of rank $d+1$.
\item For some pseudosphere arrangement $\mathcal{S} = \{S_i\}_{i \in [n]}$ in $\mathbb{S}^{d+1+k}$ such that $\displaystyle \dim \bigcap_{i \in [n]}S_i = k$, we have $\mathsf{M}_{\mathcal{S}} = \mathsf{M}$.
\item There exists an essential and centrally symmetric pseudosphere arrangement $\mathcal{S}$ in $\mathbb{S}^d$ such that its induced cell complex is shellable and $\mathsf{M} = \mathsf{M}_{\mathcal{S}}$.
\end{enumerate}
\end{theorem}

\noindent Björner et al. defined the pseudohyperplane arrangements in real projective spaces \cite[§~5.2]{BjLaStWhZi}. For us it is more convenient to define them on Euclidean spaces like Deshpande did \cite[Def~3.11]{De}. For $d \in \mathbb{N}^*$, a \textbf{pseudohyperplane} in $\mathbb{R}^{d+1}$ is set $H \subseteq \mathbb{R}^{d+1}$ such that $H$ is homeomorphic to $\mathbb{R}^{d}$, and $S = H \cap \mathbb{S}^d$ is a pseudosphere of a sphere $\mathbb{S}^d \subseteq \mathbb{R}^{d+1}$. Denote by $H^+$ and $H^-$ both connected components of $\mathbb{R}^{d+1} \setminus H$ such that $S^+ \subseteq H^+$. A \textbf{pseudohyperplane arrangement} is a finite set $\mathcal{H}$ of pseudohyperplanes in $\mathbb{R}^{d+1}$ such that $\{H \cap \mathbb{S}^d\ |\ H \in \mathcal{H}\}$ is a centrally symmetric pseudosphere arrangement in $\mathbb{S}^d$. For a pseudohyperplane arrangement $\mathcal{H} = \{H_i\}_{i \in [n]}$, define the sign functions $\epsilon_i: \mathbb{R}^{d+1} \rightarrow \mathbb{E}$ and $\epsilon_{\mathcal{H}}: \mathbb{R}^{d+1} \rightarrow \mathbb{E}^n$ respectively by $$\epsilon_i(x) := \begin{cases}
	+ & \text{if}\ x \in H_i^+ \\
	- & \text{if}\ x \in H_i^- \\
	0 & \text{if}\ x \in H_i
\end{cases} \quad \text{and} \quad \epsilon_{\mathcal{H}}(x) := \big(\epsilon_1(x), \dots, \epsilon_n(x)\big).$$
Let $\mathsf{M}_{\mathcal{H}}$ be the subset $\big\{\epsilon_{\mathcal{H}}(x)\ |\ x \in \mathbb{R}^{d+1}\big\}$ of $\mathbb{E}^n$. \emph{Denoting by $\mathcal{S}$ the pseudosphere arrangement $\{H \cap \mathbb{S}^d\ |\ H \in \mathcal{H}\}$, we then have $\mathsf{M}_{\mathcal{H}} = \mathsf{M}_{\mathcal{S}}$}. Hence, we deduce the following corollary.

\begin{corollary} \label{CoPs}
Let $d$ be a positive integer and $\mathsf{M}$ a subset of $\mathbb{E}^n$. Then the following conditions are equivalent:
\begin{enumerate}
\item $\mathsf{M}$ is an oriented matroid of rank $d$.
\item For some pseudohyperplane arrangement $\mathcal{H} = \{H_i\}_{i \in [n]}$ in $\mathbb{R}^{d+k}$ with $\displaystyle \dim \bigcap_{i \in [n]}H_i = k$, we have $\mathsf{M}_{\mathcal{H}} = \mathsf{M}$.
\item There exists a pseudohyperplane arrangement $\mathcal{H}$ in $\mathbb{R}^d$ such that $\mathsf{M} = \mathsf{M}_{\mathcal{H}}$.
\end{enumerate}
\end{corollary} 

\begin{proof}
Use the topological representation theorem, the definition of a pseudohyperplane arrangement, and the fact that, for any pseudosphere arrangement $\mathcal{S}$ in $\mathbb{S}^{d-1}$, there exists a pseudohyperplane arrangement $\mathcal{H}$ in $\mathbb{R}^d$ such that $\mathcal{S} = \{H \cap \mathbb{S}^{d-1}\ |\ H \in \mathcal{H}\}$ and $\mathsf{M}_{\mathcal{H}} = \mathsf{M}_{\mathcal{S}}$.
\end{proof}

\noindent A subset $F$ of $\mathbb{R}^d$ is a \textbf{face} of a pseudohyperplane arrangement $\mathcal{H}$ if there exists a covector $u$ in $\mathsf{M}_{\mathcal{H}}$ such that $F := \epsilon_{\mathcal{H}}^{-1}(u)$. Denote the set formed by the faces of $\mathcal{H}$ by $F_{\mathcal{H}}$. It is a poset with partial order $\preceq$ defined by $F \preceq G \, \Longleftrightarrow \, \epsilon_{\mathcal{H}}(F) \leq \epsilon_{\mathcal{H}}(G)$. \emph{We observe that $(F_{\mathcal{H}}, \preceq)$ is poset isomorphic to $(\mathsf{M}_{\mathcal{H}}, \leq)$}. A \textbf{chamber} of $\mathcal{H}$ is a face $C$ such that $\epsilon_{\mathcal{H}}(C)$ contains no $0$. Denote the set formed by the chambers of $\mathcal{H}$ by $C_{\mathcal{H}}$. Consider a pseudohyperplane arrangement $\mathcal{H} = \{H_i\}_{i \in [n]}$ in $\mathbb{R}^d$, and let $R_{\mathcal{H}}$ be the polynomial ring $\mathbb{Z}\big[a_i^+, a_i^-\ \big|\ i \in [n]\big]$ with variables $a_i^+, a_i^-$. Define the Aguiar-Mahajan distance $v: C_{\mathcal{H}} \times C_{\mathcal{H}} \rightarrow R_{\mathcal{H}}$ on chambers by $$v(C,C) = 1 \quad \text{and} \quad v(C,D) = \prod_{i \in \mathrm{S}\big(\epsilon_{\mathcal{H}}(C), \epsilon_{\mathcal{H}}(D)\big)} a_i^{\epsilon_i(C)}\,\ \text{if}\,\ C \neq D.$$
The Varchenko determinant of the pseudohyperplane arrangement $\mathcal{H}$ is $$V_{\mathcal{H}} := \big|v(D,C)\big|_{C,D \in C_{\mathcal{H}}}.$$

\noindent The weight of a face $F \in F_{\mathcal{H}} \setminus C_{\mathcal{H}}$ is the monomial $\displaystyle b_F := \prod_{\substack{i \in [n] \\ \epsilon_i(F) = 0}} a_i^+ a_i^-$, and its multiplicity the integer $\displaystyle \beta_F := \frac{\#\{C \in C_{\mathcal{H}}\ |\ \overline{C} \cap H_i = F\}}{2}$ with $\epsilon_i(F) = 0$. One can easily see that $b_F = \mathrm{b}_{\epsilon_{\mathcal{H}}(F)}$ and $\beta_F = \beta_{\epsilon_{\mathcal{H}}(F)}$ if one considers the oriented matroid $\mathsf{M}_{\mathcal{H}}$. \emph{We can now state the intermediate result which will allow to prove Theorem~\ref{ThMain}}.

\begin{theorem}  \label{main}
Let $\mathcal{H}$ be a pseudohyperplane arrangement in $\mathbb{R}^d$. Then,
$$\displaystyle V_{\mathcal{H}} = \prod_{F \in F_{\mathcal{H}} \setminus C_{\mathcal{H}}} (1 - b_F)^{\beta_F}.$$
\end{theorem}

\noindent \emph{From the isomorphism between $(F_{\mathcal{H}}, \preceq)$ and $(\mathsf{M}_{\mathcal{H}}, \leq)$, we naturally have $V_{\mathcal{H}} = V_{\mathsf{M}_{\mathcal{H}}}$}.

\begin{example}
Consider the oriented matroid $\mathsf{M} = \mathsf{M}_{\mathcal{S}}$ where $\mathcal{S}$ is the non-Pappus arrangement of $10$ pseudocircles \cite[Ex.~3.12]{De}. There exist $i \in [10]$, that we can assume to be $10$, and a covector $u \in \mathsf{M}$ such that the fiber $\mathsf{M}_{[9],u}$ is poset isomorphic to the face poset of the non-Pappus configuration having $9$ pseudolines in Figure~\ref{ExNo}. It is known that there exists no hyperplane arrangement whose face poset is isomorphic to that of the non-Pappus configuration \cite[Ex.~7.28]{Zi}. We use that latter to computer the Varchenko determinant of $\mathsf{M}_{[9],u}$. For simplicity, we assume that the variables associated to that the pseudolines are all equal to $a$. Regarding its faces, the non-Pappus configuration has $33$ chambers, then $\mathsf{M}_{[9],u}$ contains $33$ topes. Moreover, it has $43$ lines with weight $a^2$ and multiplicity $1$, $8$ points with weight $a^6$ and multiplicity $1$, and $7$ points with weight $a^4$ and multiplicity $0$. Hence we obtain $$V_{\mathsf{M}_{[9],u}} = (1-a^2)^{43} (1-a^6)^8.$$
	
\begin{figure}[h]
	\centering
	\includegraphics[scale=0.75]{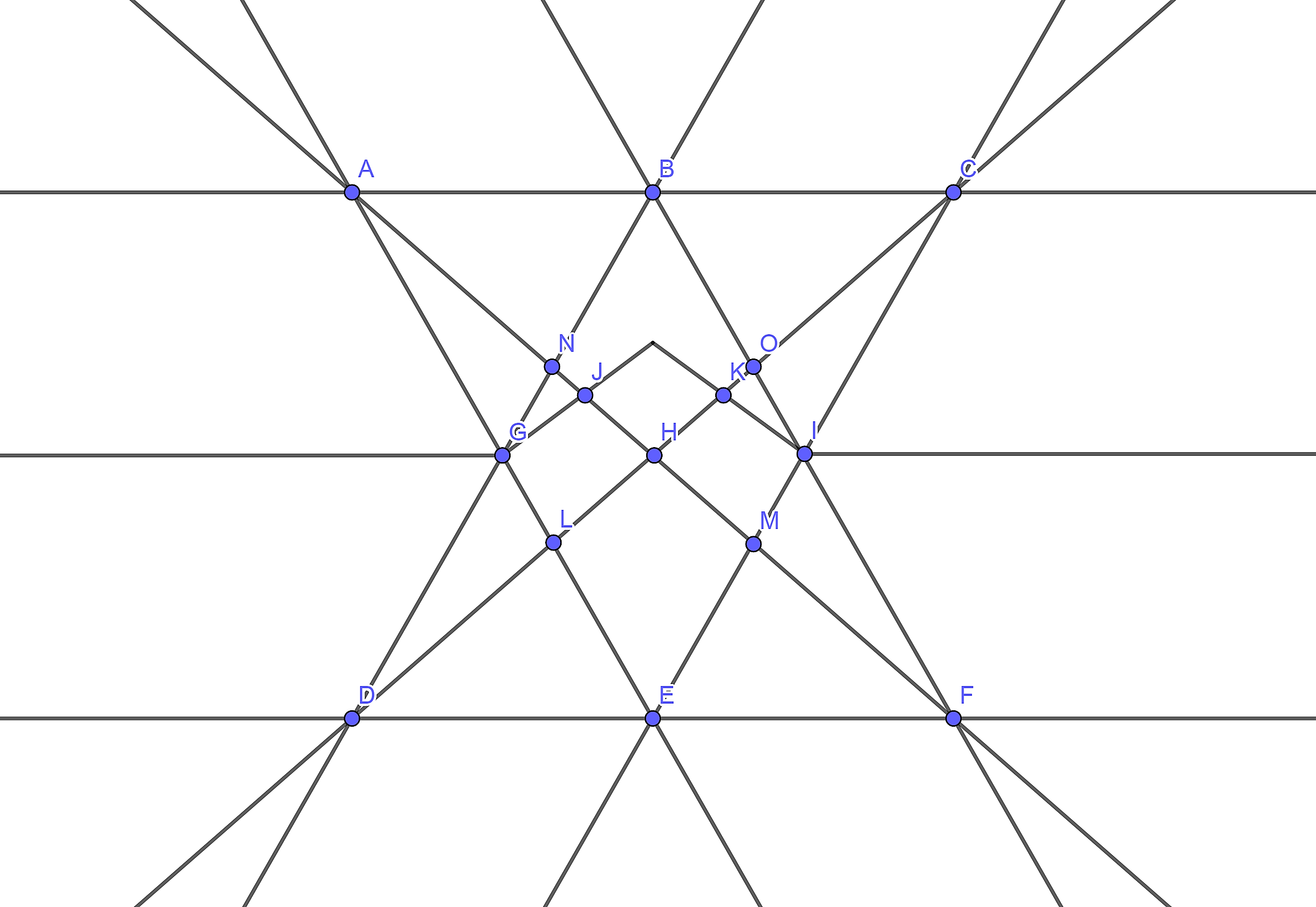}
	\caption{Non-Pappus Configuration}
	\label{ExNo}
\end{figure}	
\end{example}

\section{Generalized Witt Identity}  \label{SeWi}

\noindent We extend the Witt identity \cite[Proposition~7.30]{AgMa} to pseudohyperplane arrangements, and use that extension to prove a key equality on the chambers. The faces of a pseudohyperplane arrangement $\mathcal{H}$ form a monoid with the multiplication defined, for $F,G \in F_{\mathcal{H}}$, by
$$FG := \epsilon_{\mathcal{H}}^{-1}\big(\epsilon_{\mathcal{H}}(F) \circ \epsilon_{\mathcal{H}}(G)\big).$$ 

\noindent A \textbf{nested face} of $\mathcal{H}$ is a pair $(F,G)$ of faces in $F_{\mathcal{H}}$ such that $F \prec G$. Denote by $F_{\mathcal{H}}^{(F,G)}$ the set of faces $\{E \in F_{\mathcal{H}}\ |\ F \preceq E \preceq G\}$. Moreover, letting $\displaystyle c_{\mathcal{H}} := \dim \bigcap_{H \in \mathcal{H}}H$, the \textbf{rank} of a face $F \in F_{\mathcal{H}}$ is $\mathrm{rk}\,F := \dim F - c_{\mathcal{H}}$. The \textbf{opposite} of a face $F \in F_{\mathcal{H}}$ is the face $\tilde{F}$ of $\mathcal{H}$ such that $\epsilon_{\mathcal{H}}(\tilde{F}) = -\epsilon_{\mathcal{H}}(F)$. For a face $F$ of $\mathcal{H} = \{H_i\}_{i \in [n]}$, let $\mathcal{H}_F$ be the pseudohyperplane arrangement $\big\{H_i \in \mathcal{H}\ \big|\ \epsilon_i(F) = 0\big\}$. Let $\chi$ be the function Euler characteristic of the shape of a topological space. Furthermore, assign a variable $x_C$ to each chamber $C \in C_{\mathcal{H}}$.

\begin{proposition}  \label{Witt}
Let $\mathcal{H}$ be a pseudohyperplane arrangement in $\mathbb{R}^d$, $D \in C_{\mathcal{H}}$, and $(A,D)$ a nested face of $\mathcal{H}$. Then, $$\sum_{F \in F_{\mathcal{H}}^{(A,D)}} (-1)^{\mathrm{rk}\,F} \sum_{\substack{C \in C_{\mathcal{H}} \\ FC = D}} x_C \, = \, (-1)^{\mathrm{rk}\,D} \sum_{\substack{C \in C_{\mathcal{H}} \\ AC = A\tilde{D}}} x_C.$$
\end{proposition}

\begin{proof}
We have $\displaystyle \sum_{F \in F_{\mathcal{H}}^{(A,D)}} (-1)^{\mathrm{rk}\,F} \sum_{\substack{C \in C_{\mathcal{H}} \\ FC = D}} x_C = \sum_{C \in C_{\mathcal{H}}} \Big( \sum_{\substack{F \in F_{\mathcal{H}}^{(A,D)} \\ FC = D}} (-1)^{\mathrm{rk}\,F} \Big) x_C$. 
\begin{itemize}
\item If $\epsilon_{\mathcal{H}_A}(C) = \epsilon_{\mathcal{H}_A}(A\tilde{D})$, then $\displaystyle \sum_{\substack{F \in F_{\mathcal{H}}^{(A,D)} \\ FC = D}} (-1)^{\mathrm{rk}\,F} = (-1)^{\mathrm{rk}\,D}$.
\end{itemize}

\noindent Denote by $F_{\mathcal{H}}^{(A,\centerdot)}$ the set of faces $\{F \in F_{\mathcal{H}}\ |\ A \preceq F\} = \big\{F \in F_{\mathcal{H}}\ \big|\ \epsilon_{\mathcal{H} \setminus \mathcal{H}_A}(F) = \epsilon_{\mathcal{H} \setminus \mathcal{H}_A}(A)\big\}$. Let $f: F_{\mathcal{H}}^{(A,\centerdot)} \rightarrow F_{\mathcal{H}_A}$ be the bijection such that, if $F \in F_{\mathcal{H}}^{(A,\centerdot)}$, then $f(F)$ is the face of $\mathcal{H}_A$ such that $\epsilon_{\mathcal{H}_A}\big(f(F)\big) = \epsilon_{\mathcal{H}_A}(F)$.

\begin{itemize}
\item If $\epsilon_{\mathcal{H}_A}(C) = \epsilon_{\mathcal{H}_A}(D)$, then
\begin{align*}
\sum_{\substack{F \in F_{\mathcal{H}}^{(A,D)} \\ FC = D}} (-1)^{\mathrm{rk}\,F} & = (-1)^{-c_{\mathcal{H}}} \sum_{F \in F_{\mathcal{H}}^{(A,D)}} (-1)^{\dim F} \\
& = (-1)^{c_{\mathcal{H}}} \sum_{F \in f\big(F_{\mathcal{H}}^{(A,D)}\big)} (-1)^{\dim F} \\
& = (-1)^{c_{\mathcal{H}}} \chi\big(\overline{f(D)}\big) \\
& = 0.
\end{align*}
\item The case $\epsilon_{\mathcal{H}_A}(C) \notin \big\{\epsilon_{\mathcal{H}_A}(D), \epsilon_{\mathcal{H}_A}(A\tilde{D})\big\}$ remains. Assume that $\epsilon_{\mathcal{H}_A}(D) = (+, \dots, +)$, and define the pseudohyperplane arrangement $\mathcal{H}_A(C) := \big\{H_i \in \mathcal{H}_A\ \big|\ \epsilon_i(C) = -\big\}$. Moreover, if $E \in C_{\mathcal{H}_A(C)}$, let  $F_{\mathcal{H}_A(C)}^{(d-1,E)} := \{F \in F_{\mathcal{H}_A(C)}\ |\ F \preceq E,\, \dim F = d-1\}$. If $\#\mathcal{H}_A(C) > 1$, then $$\forall F \in F_{\mathcal{H}_A(C)}^{(d-1,E)},\, \exists F' \in F_{\mathcal{H}_A(C)}^{(d-1,E)} \setminus \{F\}:\ \dim \mathrm{int}(\overline{F} \cap \overline{F'}) = d-2.$$
We obtain,
\begin{align*}
\sum_{\substack{F \in F_{\mathcal{H}}^{(A,D)} \\ FC = D}} (-1)^{\mathrm{rk}\,F} & = (-1)^{- c_{\mathcal{H}}} \sum_{\substack{F \in F_{\mathcal{H}}^{(A,D)} \\ \epsilon_{\mathcal{H}_A(C)}(F) \, = \,  (+, \dots, +)}} (-1)^{\dim F} \\
& = (-1)^{c_{\mathcal{H}}} \sum_{\substack{F \in f\big(F_{\mathcal{H}}^{(A,D)}\big) \\ \epsilon_{\mathcal{H}_A(C)}(F) \, = \,  (+, \dots, +)}} (-1)^{\dim F} \\
& = (-1)^{c_{\mathcal{H}}} \chi\Big(\overline{f(D)} \setminus \bigcup_{F \in F_{\mathcal{H}_A(C)}^{(d-1,f(D))}} \overline{F}\Big) \\
& = (-1)^{c_{\mathcal{H}}} \Big(\chi\big(\overline{f(D)}\big) - \chi\big( \bigcup_{F \in F_{\mathcal{H}_A(C)}^{(d-1,f(D))}} \overline{F} \big)\Big) \\
& = 0.
\end{align*}
\end{itemize}
Hence $\displaystyle \sum_{F \in F_{\mathcal{H}}^{(A,D)}} (-1)^{\mathrm{rk}\,F} \sum_{\substack{C \in C_{\mathcal{H}} \\ FC = D}} x_C \, = \, (-1)^{\mathrm{rk}\,D} \sum_{\substack{C \in C_{\mathcal{H}} \\ \epsilon_{\mathcal{H}_A}(C) \, = \, \epsilon_{\mathcal{H}_A}(A\tilde{D})}} x_C \, = \, (-1)^{\mathrm{rk}\,D} \sum_{\substack{C \in C_{\mathcal{H}} \\ AC = A\tilde{D}}} x_C$.
\end{proof}

\begin{lemma}  \label{cfd}
Let $\mathcal{H}$ be a pseudohyperplane arrangement in $\mathbb{R}^d$, $C, D \in C_{\mathcal{H}}$, and $F \in F_{\mathcal{H}}$ such that $F \preceq C$. Then, $$v(C,D) = v(C,FD) \, v(FD,D).$$
\end{lemma}

\begin{proof}
Assume that $\mathcal{H} = \{H_i\}_{i \in [n]}$, and for any chambers $F,G \in C_{\mathcal{H}}$, denote the set of half-spaces containing $F$ but not $G$ by $\mathscr{H}_{F,G} := \big\{H_i^{\epsilon_i(F)}\ \big|\ i \in [n]:\, \epsilon_i(F) \neq \epsilon_i(G)\big\}$. We have
$$\mathscr{H}_{C,FD} = \big\{H_i^{\epsilon_i(C)}\ \big|\ H_i \in \mathcal{A}_F,\, \epsilon_i(C) \neq \epsilon_i(FD)\big\} = \big\{H_i^{\epsilon_i(C)}\ \big|\ H_i \in \mathcal{A}_F,\, \epsilon_i(C) \neq \epsilon_i(D)\big\},$$
and also
$$\mathscr{H}_{FD,D} = \big\{H_i^{\epsilon_i(FD)}\ \big|\ H_i \in \mathcal{A} \setminus \mathcal{A}_F,\, \epsilon_i(FD) \neq \epsilon_i(D)\big\} = \big\{H_i^{\epsilon_i(C)}\ \big|\ H_i \in \mathcal{A} \setminus \mathcal{A}_F,\, \epsilon_i(C) \neq \epsilon_i(D)\big\}.$$
Then, $\mathscr{H}_{C,D} = \mathscr{H}_{C,FD} \sqcup \mathscr{H}_{FD,D}$.
\end{proof}

\noindent The module of $R_{\mathcal{H}}$-linear combinations of chambers in $C_{\mathcal{H}}$ is
$\displaystyle M_{\mathcal{H}} := \Big\{\sum_{C \in C_{\mathcal{H}}} x_C C\ \Big|\ x_C \in R_{\mathcal{H}}\Big\}$.

\noindent Let $\{C^*\}_{C \in C_{\mathcal{H}}}$ be the dual basis of the basis $C_{\mathcal{H}}$. Define the linear map $\gamma_{\mathcal{H}}: M_{\mathcal{H}} \rightarrow M_{\mathcal{H}}^*$, for every $D \in C_{\mathcal{H}}$, by
$$\gamma_{\mathcal{H}}(D) := \sum_{C \in C_{\mathcal{H}}} v(D,C)\, C^*.$$

\noindent Let $B_{\mathcal{H}}$ be the extension ring $\displaystyle \bigg\{\frac{p}{\displaystyle \prod_{F \in F_{\mathcal{H}}  \setminus C_{\mathcal{H}}} (1 - b_F)^{k_F}}\ \bigg|\ p \in R_{\mathcal{H}},\, k_F \in \mathbb{N}\bigg\}$ of $R_{\mathcal{H}}$.

\begin{proposition} \label{D*}
Let $\mathcal{H}$ be a pseudohyperplane arrangement in $\mathbb{R}^d$, and $D \in C_{\mathcal{H}}$. Then,
$$D^* = \sum_{C \in C_{\mathcal{H}}} x_C \, \gamma_{\mathcal{H}}(C) \quad \text{with} \quad x_C \in B_{\mathcal{H}}.$$
\end{proposition}

\begin{proof}
The proof is similar to the proof of \cite[Proposition~8.13]{AgMa}. For a nested face $(A,D)$, let
$$\displaystyle m(A,D) := \sum_{\substack{C \in C_{\mathcal{H}} \\ AC = D}} v(D,C)\, C^* \in M_{\mathcal{H}}^*.$$
One proves by backward induction that $\displaystyle m(A,D) = \sum_{C \in C_{\mathcal{H}}} x_C \, \gamma_{\mathcal{H}}(C)$ with $x_C \in B_{\mathcal{H}}$. We obviously have $m(D,D) = \gamma_{\mathcal{H}}(D)$. Then, Proposition~\ref{Witt} applied to $x_C = v(D,C)\, C^*$ in addition to Lemma~\ref{cfd} yield
$$\sum_{F \in F_{\mathcal{H}}^{(A,D)}} (-1)^{\mathrm{rk}\,F} m(F,D) = (-1)^{\mathrm{rk}\,D} \sum_{\substack{C \in C_{\mathcal{H}} \\ AC = A\tilde{D}}} v(D,C)\, C^* = (-1)^{\mathrm{rk}\,D} \, v(D, A\tilde{D}) \, m(A, A\tilde{D}).$$
Hence, $\displaystyle m(A,D) - (-1)^{\mathrm{rk}\,D - \mathrm{rk}\,A} \, v(D, A\tilde{D}) \, m(A, A\tilde{D}) = \sum_{F \in F_{\mathcal{H}}^{(A,D)} \setminus \{A\}} (-1)^{\mathrm{rk}\,F - \mathrm{rk}\,A +1} m(F,D)$.
By induction hypothesis, for every $C \in C_{\mathcal{H}}$, there exists $a_C \in B_{\mathcal{H}}$, such that $$\sum_{F \in F_{\mathcal{H}}^{(A,D)} \setminus \{A\}} (-1)^{\mathrm{rk}\,F - \mathrm{rk}\,A +1} m(F,D) = \sum_{C \in C_{\mathcal{H}}} a_C \, \gamma_{\mathcal{H}}(C).$$
Since $A \preceq A\tilde{D}$ and $A(\widetilde{A\tilde{D}}) = D$, by replacing $D$ with $A\tilde{D}$, there exists also $e_C \in B_{\mathcal{H}}$ for every $C \in C_{\mathcal{H}}$ such that $\displaystyle m(A, A\tilde{D}) - (-1)^{\mathrm{rk}\,A\tilde{D} - \mathrm{rk}\,A} \, v(A\tilde{D}, D) \, m(A,D) = \sum_{C \in C_{\mathcal{H}}} e_C \, \gamma_{\mathcal{H}}(C)$.
Therefore, $$m(A,D) = \sum_{C \in C_{\mathcal{H}}} \frac{a_C + (-1)^{\mathrm{rk}\,D - \mathrm{rk}\,A} \, v(D, A\tilde{D}) \, e_C}{1 - b_A} \gamma_{\mathcal{H}}(C).$$
Hence, $\displaystyle D^* = m\Big(\bigcap_{H \in \mathcal{H}}H,\, D\Big) = \sum_{C \in C_{\mathcal{H}}} x_C \, \gamma_{\mathcal{H}}(C)$ with $x_C \in B_{\mathcal{H}}$.
\end{proof}

\section{Determinant Computing}  \label{SeVa}

\noindent We finally compute the Varchenko determinant of a pseudohyperplane arrangement by first proving that it has the form $\displaystyle \prod_{F \in F_{\mathcal{H}} \setminus C_{\mathcal{H}}} (1 - b_F)^{l_F}$, and by determining $l_F$ for each face $F$. The \textbf{Varchenko matrix} of a pseudohyperplane arrangement $\mathcal{H}$ is $\mathbf{M}_{\mathcal{H}} := \big(v(D,C)\big)_{C,D \in C_{\mathcal{H}}}$.

\begin{proposition} \label{cen}
Let $\mathcal{H}$ be a pseudohyperplane arrangement in $\mathbb{R}^d$. For a face $F \in F_{\mathcal{H}} \setminus C_{\mathcal{H}}$, there exists a nonnegative integer $l_F$ such that $$\det \mathbf{M}_{\mathcal{H}} = \prod_{F \in F_{\mathcal{H}} \setminus C_{\mathcal{H}}} (1 - b_F)^{l_F}.$$ 
\end{proposition}

\begin{proof}
It is clear that $\mathbf{M}_{\mathcal{H}}$ is the matrix representation of $\gamma_{\mathcal{H}}$. The determinant of $\mathbf{M}_{\mathcal{H}}$ is a polynomial in $R_{\mathcal{H}}$ with constant term $1$, so $\mathbf{M}_{\mathcal{H}}$ is invertible. From Proposition~\ref{D*}, we know that, for every chamber $D \in C_{\mathcal{H}}$, there exist $x_C \in B_{\mathcal{H}}$ such that $\displaystyle D^* = \sum_{C \in C_{\mathcal{H}}} x_C \gamma_{\mathcal{H}}(C)$. Hence, $$\gamma_{\mathcal{H}}^{-1}(D^*) = \sum_{C \in C_{\mathcal{H}}} x_C \, C \quad \text{with} \quad x_C \in B_{\mathcal{H}}.$$
As the matrix representation of $\gamma_{\mathcal{H}}^{-1}$ is $\mathbf{M}_{\mathcal{H}}^{-1}$, each entry of $\mathbf{M}_{\mathcal{H}}^{-1}$ is then an element of $B_{\mathcal{H}}$. To finish, note that $\displaystyle \mathbf{M}_{\mathcal{H}}^{-1} = \frac{\mathrm{adj}(\mathbf{M}_{\mathcal{H}})}{\det \mathbf{M}_{\mathcal{H}}}$, where each entry of $\mathrm{adj}(\mathbf{M}_{\mathcal{H}})$ is a polynomial in $R_{\mathcal{H}}$. Then, the only possibility is $\det \mathbf{M}_{\mathcal{H}}$ has the form $\displaystyle k \prod_{F \in F_{\mathcal{H}} \setminus C_{\mathcal{H}}} (1 - b_F)^{l_F}$, with $k \in \mathbb{Z}$. As the constant term of $\det \mathbf{M}_{\mathcal{H}}$ is $1$, we deduce that $k=1$.
\end{proof}

\noindent Let $\mathcal{K}$ be a subset of a pseudohyperplane arrangement $\mathcal{H}$. Define an \textbf{apartment} of $\mathcal{H}$ to be a chamber of $\mathcal{K}$. Denote the set formed by the apartments of $\mathcal{H}$ by $K_{\mathcal{H}}$. The \textbf{restriction} of $\mathcal{H}$ to $K \in K_{\mathcal{H}}$ is the pseudohyperplane arrangement $\mathcal{H}^K = \{H_i \in \mathcal{H}\ |\ H_i \cap K \neq \emptyset\}$. The sets formed by the faces and the chambers in $K$ are $F_{\mathcal{H}}^K := \{F \in F_{\mathcal{H}}\ |\ F \subseteq K\}$ and $C_{\mathcal{H}}^K := C_{\mathcal{H}} \cap F_{\mathcal{H}}^K$ respectively. Moreover, the Varchenko matrix of $\mathcal{H}$ restricted to $C_{\mathcal{H}}^K$ is
$$\mathbf{M}_{\mathcal{H}}^K := \big(v(D,C)\big)_{C,D \in C_{\mathcal{H}}^K}.$$

\begin{corollary} 
Let $\mathcal{H}$ a pseudohyperplane arrangement in $\mathbb{R}^d$, and $K \in K_{\mathcal{H}}$. Then, $$\det \mathbf{M}_{\mathcal{H}}^K = \prod_{F \in F_{\mathcal{H}}^K \setminus C_{\mathcal{H}}^K} (1- b_F)^{l_F}.$$ 
\end{corollary}

\begin{proof}
Set $a_i^+ = a_i^- = 0$ for every $H_i \in \mathcal{H} \setminus \mathcal{H}^K$. Then, $$v(C,D) \neq 0 \quad \Longleftrightarrow \quad C=D \, \ \text{or} \, \ C,D \in C_{\mathcal{H}}^K \, \ \text{or} \, \ C,D \in C_{\mathcal{H}}^{\tilde{K}}.$$
Hence, for a suitable chamber indexing of the rows and the columns, we obtain $$\mathbf{M}_{\mathcal{H}} = \mathbf{M}_{\mathcal{H}}^K \oplus \mathbf{M}_{\mathcal{H}}^{\tilde{K}} \oplus \mathbf{I},$$ where $\mathbf{I}$ is the identity matrix of order $\#C_{\mathcal{H}} - 2\,\#C_{\mathcal{H}}^K$. Using Proposition~\ref{cen}, we get
\begin{align*}
\det \mathbf{M}_{\mathcal{H}}^K \oplus \mathbf{M}_{\mathcal{H}}^{\tilde{K}} & = \prod_{F \in F_{\mathcal{H}}^K \setminus C_{\mathcal{H}}^K} (1- b_F)^{l_F} \ \prod_{G \in F_{\mathcal{H}}^{\tilde{K}} \setminus C_{\mathcal{H}}^{\tilde{K}}} (1- b_G)^{l_G} \\
& = \prod_{F \in F_{\mathcal{H}}^K \setminus C_{\mathcal{H}}^K} (1- b_F)^{l_F} (1- b_{\tilde{F}})^{l_{\tilde{F}}} \\
& = \prod_{F \in F_{\mathcal{H}}^K \setminus C_{\mathcal{H}}^K} (1- b_F)^{l_F + l_{\tilde{F}}}. 
\end{align*}
Since $\mathbf{M}_{\mathcal{H}}^K = \mathbf{M}_{\mathcal{H}}^{\tilde{K}}$ for a suitable row and column indexing, then $l_F = l_{\tilde{F}}$.
\end{proof}

\noindent For $F \in F_\mathcal{H}$, and $H_i \in \mathcal{H}$ such that $F \subseteq H_i$, let $\displaystyle \beta_F^{H_i} := \frac{\#\{C \in C_{\mathcal{H}}\ |\ \overline{C} \cap H_i = F\}}{2}$.

\begin{theorem} \label{V}
Let $\mathcal{H}$ be a pseudohyperplane arrangement in $\mathbb{R}^d$, and $F \in F_{\mathcal{H}}$. Then, $\beta_F^{H_i}$ has the same value $\beta_F$ for every $H_i \in \mathcal{H}_F$, and $$\det \mathbf{M}_{\mathcal{H}} = \prod_{F \in F_{\mathcal{H}} \setminus C_{\mathcal{H}}} (1 - b_F)^{\beta_F}.$$
\end{theorem}

\begin{proof}
Take a face $E \in F_{\mathcal{H}} \setminus C_{\mathcal{H}}$. There exists an apartment $K \in K_{\mathcal{H}}$ such that $$\displaystyle E = \bigcap_{H_i \in \mathcal{H}^K} H_i \cap K.$$
We prove by backward induction on the dimension of $E$ that $$\forall H_i,H_j \in \mathcal{H}_E:\ \beta_E^{H_i} = \beta_E^{H_j} = \beta_E \quad \text{and} \quad \det \mathbf{M}_{\mathcal{H}}^K = \prod_{F \in F_{\mathcal{H}}^K \setminus C_{\mathcal{H}}^K} (1 - b_F)^{\beta_F}.$$
Remark that $\displaystyle \beta_F^{H_i} = \frac{\#\{C \in C_{\mathcal{H}}^K\ |\ \overline{C} \cap H_i = F\}}{2}$. It is clear that, if $\dim E = d-1$, then
$$\beta_E = 1 \quad \text{and} \quad \det \mathbf{M}_{\mathcal{H}}^K = 1 - b_E.$$
If $\dim E < d-1$, by induction hypothesis, $$\det \mathbf{M}_{\mathcal{H}}^K = (1 - b_E)^{l_E} \, \prod_{F \, \in \, (F_{\mathcal{H}}^K \setminus C_{\mathcal{H}}^K) \setminus \{E\}} (1 - b_F)^{\beta_F}.$$
The leading monomial in $\det \mathbf{M}_{\mathcal{H}}^K$ is $\displaystyle (-1)^{\frac{\#C_{\mathcal{H}}^K}{2}} \prod_{C \in C_{\mathcal{H}}^K} \mathrm{v}(C, E\tilde{C}) = \big(- \prod_{H_i \in \mathcal{H}_E} a_i^+ a_i^- \big)^{\frac{\#C_{\mathcal{H}}^K}{2}}$. Then, comparing the exponent of $a_i^+ a_i^-$, we get $\displaystyle l_E \, = \, \frac{\#C_{\mathcal{H}}^K}{2} - \sum_{\substack{F \, \in \, (F_{\mathcal{H}}^K \setminus C_{\mathcal{H}}^K) \setminus \{E\} \\ F \subseteq H_i}} \beta_F^{H_i} \, = \, \beta_E^{H_i}$.
\end{proof}

\section{Proof of Theorem~\ref{ThMain}}   \label{SeMain}

\noindent From Corollary~\ref{CoPs} and  Theorem~\ref{main}, we deduce that for any oriented matroid $\mathsf{M}$ of rank $d$ included in $\mathbb{E}^n$, there exists a pseudohyperplane arrangement $\mathcal{H} = \{H_i\}_{i \in [n]}$ in $\mathbb{S}^d$ such that $V_{\mathsf{M}} = \det \mathbf{M}_{\mathcal{H}}$. \emph{The corresponding apartment and restriction to a fiber $\mathsf{M}_{I,u}$ are respectively $$K = \epsilon_{\mathcal{H}}^{-1}\Big(\big\{v \in \mathsf{M}\ |\ \forall i \in [n] \setminus I:\, \epsilon_i(v) = \epsilon_i(u)\big\}\Big) \quad \text{and} \quad \mathcal{H}^K = \{H_i\}_{i \in I}$$
which means that $(\mathsf{M}_{I,u}, \leq)$ is poset isomorphic to $(F_{\mathcal{H}}^K, \preceq)$}. Assuming that $a_i^+ = a_i^- = 0$ if $i \in [n] \setminus I$, for a suitable row and column indexing we get $\mathbf{M}_{\mathcal{H}} = \mathbf{M}_{\mathcal{H}}^K \oplus \mathbf{M}_{\mathcal{H}}^K \oplus \mathbf{I}$, where $\mathbf{I}$ is the identity matrix of order $\#C_{\mathcal{H}} - 2\,\#C_{\mathcal{H}}^K$. Since $V_{\mathsf{M}_{I,u}} = \det \mathbf{M}_{\mathcal{H}}^K$, we finally obtain
$$V_{\mathsf{M}_{I,u}} = \prod_{F \in F_{\mathcal{H}}^K \setminus C_{\mathcal{H}}} (1 - b_F)^{2\beta_F} = \prod_{v \in \mathsf{\mathsf{M}_{I,u}} \setminus T_{\mathsf{M}}} (1 - \mathrm{b}_v)^{2\beta_v} = V_{\mathsf{M}_{I,u}}^2.$$
\begin{flushright} $\blacksquare$ \end{flushright}

\bibliographystyle{abbrvnat}

\end{document}